\documentclass[10pt]{amsart}

\usepackage{latexsym}
\usepackage{amssymb}
\usepackage{mathrsfs}
\usepackage{amsmath}
\usepackage{fancybox,color}
\usepackage{enumerate}
\usepackage[active]{srcltx}
\usepackage[colorlinks]{hyperref}
\usepackage{hypernat}
\usepackage{graphicx}
\usepackage{color}
\usepackage{caption}
\usepackage{subcaption}
\usepackage{tikz}

\def\vint_#1{\mathchoice%
          {\mathop{\kern 0.2em\vrule width 0.6em height 0.69678ex depth -0.58065ex
                  \kern -0.8em \intop}\nolimits_{\kern -0.4em#1}}%
          {\mathop{\kern 0.1em\vrule width 0.5em height 0.69678ex depth -0.60387ex
                  \kern -0.6em \intop}\nolimits_{#1}}%
          {\mathop{\kern 0.1em\vrule width 0.5em height 0.69678ex
              depth -0.60387ex
                  \kern -0.6em \intop}\nolimits_{#1}}%
          {\mathop{\kern 0.1em\vrule width 0.5em height 0.69678ex depth -0.60387ex
                  \kern -0.6em \intop}\nolimits_{#1}}}
\def\vintslides_#1{\mathchoice%
          {\mathop{\kern 0.1em\vrule width 0.5em height 0.697ex depth -0.581ex
                  \kern -0.6em \intop}\nolimits_{\kern -0.4em#1}}%
          {\mathop{\kern 0.1em\vrule width 0.3em height 0.697ex depth -0.604ex
                  \kern -0.4em \intop}\nolimits_{#1}}%
          {\mathop{\kern 0.1em\vrule width 0.3em height 0.697ex depth -0.604ex
                  \kern -0.4em \intop}\nolimits_{#1}}%
          {\mathop{\kern 0.1em\vrule width 0.3em height 0.697ex depth -0.604ex
                  \kern -0.4em \intop}\nolimits_{#1}}}

 \def\1{\raisebox{2pt}{\rm{$\chi$}}}

\newtheorem{theorem}{Theorem}[section]

\newtheorem{lemma}[theorem]{Lemma}

\newtheorem{definition}[theorem]{Definition}
\newtheorem{remark}[theorem]{Remark}
\newtheorem{example}[theorem]{Example}

\newcommand{\R}{{\mathbb R}}

 \def\1{\raisebox{2pt}{\rm{$\chi$}}}
 
\def\vint_#1{\mathchoice%
          {\mathop{\kern 0.2em\vrule width 0.6em height 0.69678ex depth -0.58065ex
                  \kern -0.8em \intop}\nolimits_{\kern -0.4em#1}}%
          {\mathop{\kern 0.1em\vrule width 0.5em height 0.69678ex depth -0.60387ex
                  \kern -0.6em \intop}\nolimits_{#1}}%
          {\mathop{\kern 0.1em\vrule width 0.5em height 0.69678ex
              depth -0.60387ex
                  \kern -0.6em \intop}\nolimits_{#1}}%
          {\mathop{\kern 0.1em\vrule width 0.5em height 0.69678ex depth -0.60387ex
                  \kern -0.6em \intop}\nolimits_{#1}}}
\def\vintslides_#1{\mathchoice%
          {\mathop{\kern 0.1em\vrule width 0.5em height 0.697ex depth -0.581ex
                  \kern -0.6em \intop}\nolimits_{\kern -0.4em#1}}%
          {\mathop{\kern 0.1em\vrule width 0.3em height 0.697ex depth -0.604ex
                  \kern -0.4em \intop}\nolimits_{#1}}%
          {\mathop{\kern 0.1em\vrule width 0.3em height 0.697ex depth -0.604ex
                  \kern -0.4em \intop}\nolimits_{#1}}%
          {\mathop{\kern 0.1em\vrule width 0.3em height 0.697ex depth -0.604ex
                  \kern -0.4em \intop}\nolimits_{#1}}}

\newcommand{\aveint}[2]{\mathchoice%
          {\mathop{\kern 0.2em\vrule width 0.6em height 0.69678ex depth -0.58065ex
                  \kern -0.8em \intop}\nolimits_{\kern -0.45em#1}^{#2}}%
          {\mathop{\kern 0.1em\vrule width 0.5em height 0.69678ex depth -0.60387ex
                  \kern -0.6em \intop}\nolimits_{#1}^{#2}}%
          {\mathop{\kern 0.1em\vrule width 0.5em height 0.69678ex depth -0.60387ex
                  \kern -0.6em \intop}\nolimits_{#1}^{#2}}%
          {\mathop{\kern 0.1em\vrule width 0.5em height 0.69678ex depth -0.60387ex
                  \kern -0.6em \intop}\nolimits_{#1}^{#2}}}

\newcommand{\dist}{\operatorname{dist}}

\begin{document}

\title[A lower bound for the principal eigenvalue]
{A lower bound for the principal eigenvalue of fully nonlinear elliptic operators}

\author[P. Blanc]{Pablo Blanc}

\address{P. Blanc
\hfill\break\indent Dpto. de Matem{\'a}ticas, FCEyN,
Universidad de Buenos Aires, \hfill\break\indent 1428, Buenos Aires,
Argentina. }
\email{{\tt pblanc@dm.uba.ar,} }

\keywords{principal eigenvalue, lower bounds, fully nonlinear elliptic PDEs
\indent 2010 {\it Mathematics Subject Classification.} 35P15, 35P30, 35J60, 35J70}

\begin{abstract} 
In this article we present a new technique to obtain a lower bound for the principal Dirichlet eigenvalue of a fully nonlinear elliptic operator.
We ilustrate the construction of an appropriate radial function required to obtain the bound in several examples.
In particular we use our results to prove that $\lim_{p\to \infty}\lambda_{1,p}=\lambda_{1,\infty}=\left(\frac{\pi}{2R}\right)^2$ where $\lambda_{1,p}$ and $\lambda_{1,\infty}$ are the principal eigenvalue for the homogeneous $p$-laplacian and the homogeneous infinity laplacian respectively.
\end{abstract}

\maketitle

\section{Introduction}
\label{Sect.intro}
\setcounter{equation}{0}

Let $\Omega\subset\R^n$ be a domain and $Lu = F (u,\nabla u,D^2 u)$ a differential operator.
We consider the Dirichlet eigenvalue problem
\begin{equation}
\label{eq.eigen}
\begin{cases}
Lu+\lambda u=0 \quad &\text{in } \Omega
\\
u=0 \quad &\text{on }\partial\Omega.
\end{cases}
\end{equation}
We are interested in the principal eigenvalue of $-L$, that is the smallest number $\lambda\in\R$ for which the Dirichlet eigenvalue problem \eqref{eq.eigen} has a non-trivial solution.
Our goal here is to introduce a novel technique to obtain a lower bound for this value.

Along the whole paper we will consider solutions in the viscosity sense (see \cite{crandall1992user}), this will allow us to consider fully nonlinear operators like $Lu = F (u,\nabla u,D^2 u)$.
In this general framework we define the principal eigenvalue through the maximum principle as in \cite{berestycki1994principal}.
That is, we let
\[
\lambda_1(\Omega) = \sup\{ \lambda\in\R : \exists v \in C(\Omega) \text{ satisfying } v(x) > 0\  \forall x \in\Omega \text{ and }L v+\lambda v \leq 0 \}.
\]
This definition allows us to consider operators in non-divergence form.
In \cite{berestycki1994principal} the authors proved that for uniformly elliptic linear operators
the value $\lambda_1(\Omega)$ defined above is indeed the principal eigenvalue of $-L$.
This work opened the path to develop an eigenvalue theory for nonlinear operators. 

Let us mention some previous work that deal with the operators that we will consider as examples to illustrate our general result.
The Pucci extremal operators were studied in \cite{busca2005nonlinear}.
In \cite{birindelli2004comparison,birindelli2006} it is proved that the number defined above is the principal eigenvalue for a class of homogeneous fully nonlinear operators which includes the homogeneous $p$-laplacian (see also \cite{kawohl2016geometry} and \cite{martinez2014limit}). In \cite{juutinen2007principal} this was done for the homogeneous infinity laplacian.
The eigenvalue problem that arises as limit of the problem for the $p$-laplacian was considered in \cite{juutinen1999eigenvalue}.
Other questions were addressed in more recent work as problems in non-smooth domains \cite{birindelli2009eigenvalue}, unbounded domains \cite{berestycki2015generalizations} and simplicity of the first eigenvalue \cite{birindelli2010regularity}.

The lower bound that we obtain in this article depends on the largest radius of a ball included in $\Omega$. 
We define
\[
R=\max_{x\in\bar\Omega}\dist(x,\Omega^c).
\]

From the definition of $\lambda_1$ it is clear that the first eigenvalue is monotone with respect to the domain, that is
\[
\Omega_1 \subset \Omega_2 \Rightarrow  \lambda_1(\Omega_2) \leq \lambda_1(\Omega_1).
\]
Then
\[
\lambda_1(\Omega)\leq \lambda_1(B_R)
\]
and hence we can obtain an upper bound for the principal eigenvalue by computing this value for a ball.
We can do this by constructing a radial positive eigenfunction.
Therefore, we have to provide a radial solution $\phi(r)$ to the equation \eqref{eq.eigen} such that $\phi(R)=0$ and $\phi'(0)=0$. 
The eigenfunction will look like the one in Figure~\ref{fig:eigenfunction}.
In this way we can obtain an upper bound for the principal eigenvalue by solving certain ODE.

\begin{figure}
    \centering
    \begin{subfigure}[b]{0.5\textwidth}
        \includegraphics[width=\textwidth]{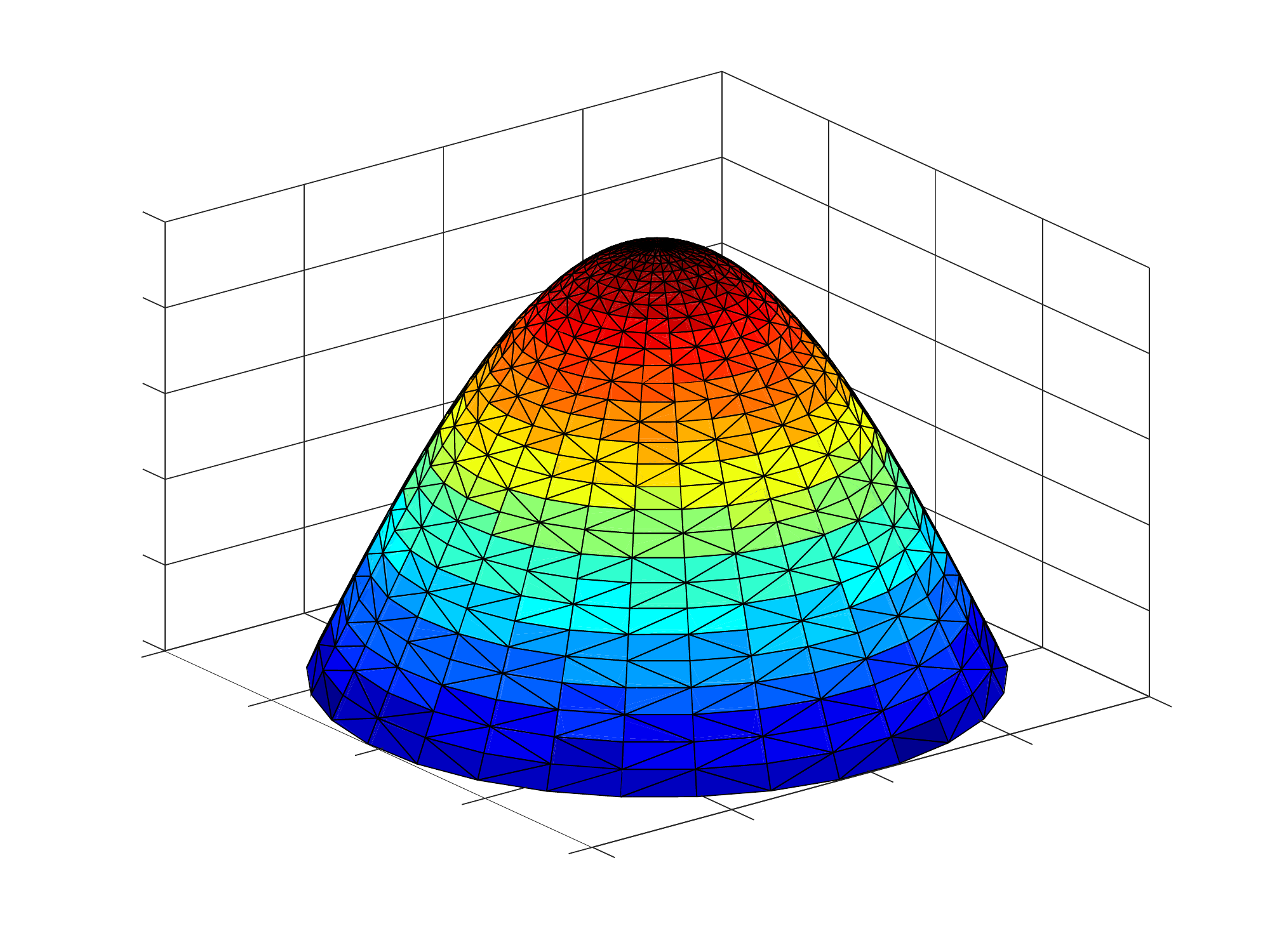}
        \caption{A eigenfunction in a ball.\\ \ }
        \label{fig:eigenfunction}
    \end{subfigure}
    ~ 
    \begin{subfigure}[b]{0.5\textwidth}
        \includegraphics[width=\textwidth]{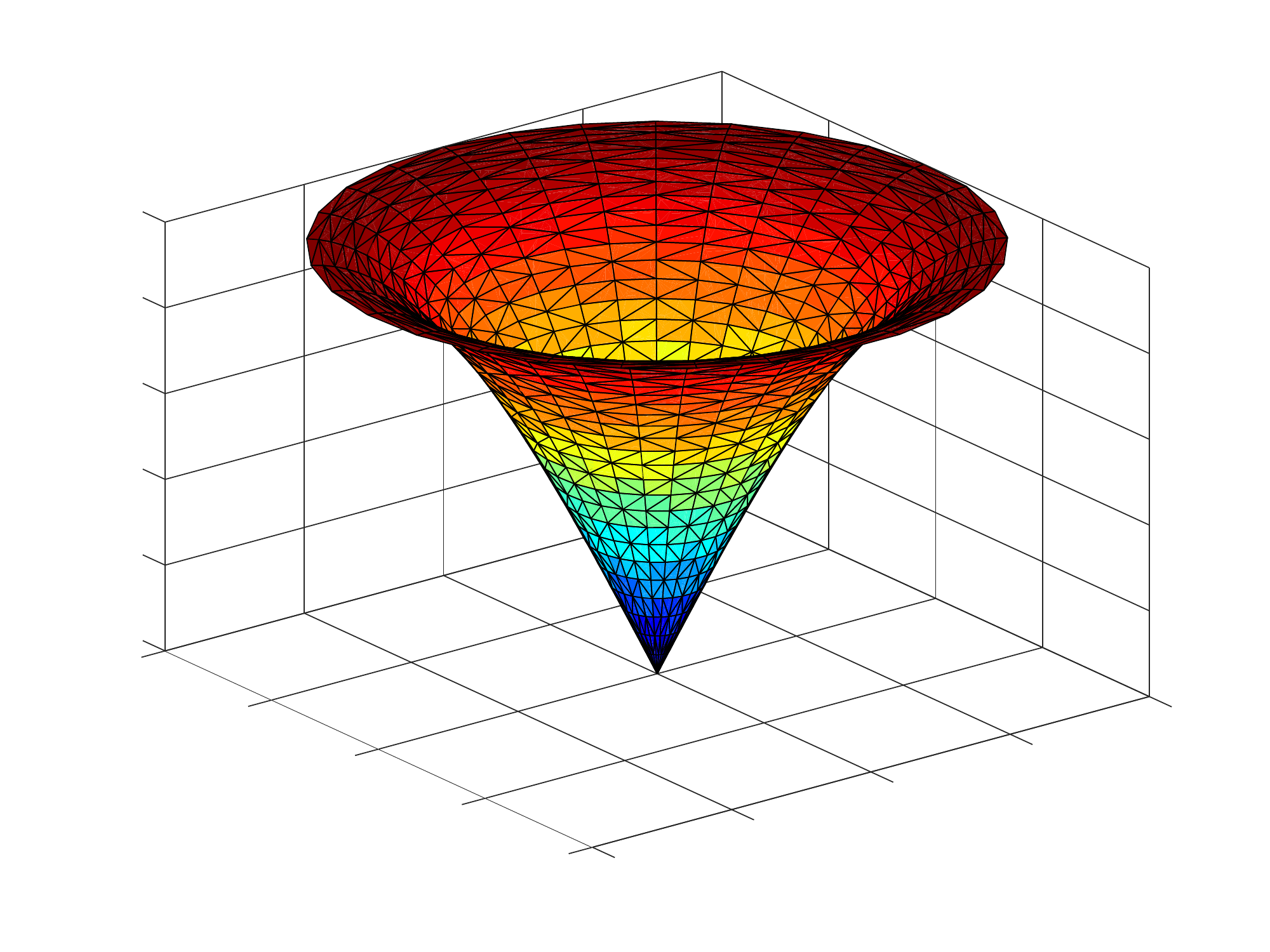}
        \caption{The radial function required in the main theorem.}
        \label{fig:radial}
    \end{subfigure}
    \caption{Radial functions that allow us to obtain bounds for the principal eigenvalue}
    \label{fig:functions}
\end{figure}

Our main result provides an analogous construction to obtain a lower bound for the principal eigenvalue.
This time we require a radial solution $\phi(r)$ to the equation $Lu+\lambda u=0$ defined in the punctured ball $B_R\setminus\{0\}$ such that $\phi'(R)=0$ and $\phi(0)=0$. 
The function will look like the one shown in Figure~\ref{fig:radial}.
In this way we can obtain a lower bound for the principal eigenvalue by solving an ODE.
The lower bound will be the value of $\lambda$ for which we can solve the ODE.

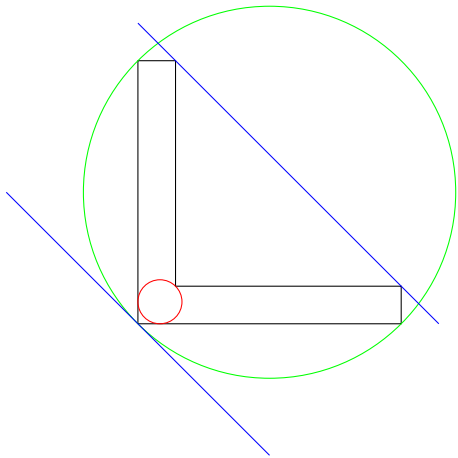
\begin{figure}
    \centering
\begin{tikzpicture}[scale=0.5]
\draw (0,0) -- (7,0) -- (7,1) -- (1,1) -- (1,7) -- (0,7) -- (0,0);
\draw [red] ({1/(1+1/sqrt(2))},{1/(1+1/sqrt(2))}) circle ({1/(1+1/sqrt(2))});
\draw [green] (3.5,3.5) circle ({3.5*sqrt(2)});
\draw [blue](3.5,-3.5) -- (-3.5,3.5);
\draw [blue] (8,0) -- (0,8);
\end{tikzpicture}
\caption{In black an L shaped domain, in red the ball of maximum radius contained in the domain, in green the ball of minimum radius that contains the domain and in blue the boundary of the narrowest strip that contains the domain.}
\label{fig:Lshape}
\end{figure}

Since our bound only depends on $R$, our technique is well suited for example for L shaped domains where considering a ball or a strip that contains $\Omega$ gives poorer results or can't be done for example if the L shaped domain is unbounded.
If we consider the L shaped domain $\Omega=\{(x,y)\in\R^2: 0\leq x,y\leq \ell \text{ and } \min\{x,y\}\leq 1\}$, we have $R=\frac{1}{1+\frac{1}{\sqrt 2}}$ but the radius of a ball and the width of a strip that contains $\Omega$ grows linearly with $\ell$, see Figure~\ref{fig:Lshape}.

We also compare our result with the classical Rayleigh-Faber-Krahn inequality in Example~\ref{nosol0}.
Even more, our technique is well suited to obtain sharp bounds for certain operators as will be shown in the examples section.

In the next section we state and prove our main result and then we outline some extensions.
Later, in Section \ref{Sect.examples}, we compute the bound explicitly for the homogeneous infinity laplacian, for the homogeneous $p$-laplacian and for other operators.
We prove that for the homogeneous infinity laplacian the principal eigenvalue is $\lambda_{1,\infty}=\left(\frac{\pi}{2R}\right)^2$.
In addition, our bound for the homogeneous $p$-laplacian proves that $\lim_{p\to\infty} \lambda_{1,p}= \lambda_{1,\infty}$.

\section{Main theorem}
\label{Sect.main}
\setcounter{equation}{0}

Let $\Omega\subset\R^n$ be a domain (not necessarily bounded) and $Lu := F (u,\nabla u,D^2 u)$ a fully nonlinear operator.
Here ${F:\R\times (\R^n-\{0\})\times S^{n\times n}\to\R}$ where $S^{n\times n}$ denotes the set of real $n\times n$ symmetric matrices.
As we are interested in operators like the homogeneous infinity laplacian and $p$-laplacian which are not well defined where the gradient vanishes we will give a suitable definition of solution that includes these operators. 

Let us recall the definition of viscosity solution.
Since $F$ may not be continuous when the gradient vanishes we need to consider the lower semicontinous $F_*$ and upper semicontinous $F^*$ envelopes of $F$. That is,
\[
F^*(s,q,Y)=\limsup_{(r,p,X)\to (s,q,Y)}F(r,p,X)
\]
and
\[
F_*(s,q,Y)=\liminf_{(r,p,X)\to (s,q,Y)}F(r,p,X).
\]

\begin{definition} \label{def.sol.viscosa}\rm
We consider the equation
\[
F (u,\nabla u,D^2 u)+\lambda u = 0.
\]
\begin{enumerate}
\item A lower semi-continuous function $ u $ is a viscosity
supersolution if for every $ \psi \in C^2$ such that $ \psi $
touches $u$ at $x \in \Omega$ strictly from below, we have
$$F_*(\psi (x),\nabla \psi (x), D^2 \psi (x))+\lambda\psi(x)\leq 0.$$

\item An upper semi-continuous function $u$ is a subsolution if
for every $ \psi \in C^2$ such that $\psi$ touches $u$ at $ x \in
\Omega$ strictly from above, we have
$$F^*(\psi (x),\nabla \psi (x), D^2 \psi (x))+\lambda\psi(x)\geq 0.$$

\item Finally, $u$ is a viscosity solution if it is both a sub- and supersolution.
\end{enumerate}
\end{definition}

\begin{remark}\rm
We have given above a definition of viscosity solution that is well suited for the equations that we will treat in the next section.
The definition can be slightly different depending on the context, see Section 9 in \cite{crandall1992user}.
We want to remark that these differences have no effect in our results.
\end{remark}

As we have mentioned in the introduction, we want to obtain a lower bound for the principal eigenvalue of $-L$ given by
\[
\lambda_1(\Omega) = \sup\{ \lambda\in\R : \exists v \in C(\Omega) \text{ satisfying } v(x) > 0\  \forall x \in\Omega \text{ and }L v+\lambda v \leq 0 \},
\]
where the last inequality holds in the viscosity sense. 
Let us recall that here
\[
R=\max_{x\in\bar\Omega}\dist(x,\Omega^c).
\]

We are ready to state and prove the main theorem of this paper.

\begin{theorem}
\label{thm.main}
Suppose $\phi(r)$ is an increasing radial function defined in $B_{r}$  for some $r>R$ with $\phi(0)=0$ and $\lambda\in\R$  is such that
\[
L\phi + \lambda \phi \leq  0
\]
in $B_r\setminus\{0\}$. Then $\lambda_1(\Omega)\geq\lambda$.
\end{theorem}

\begin{figure}
    \centering
    \includegraphics[scale=0.5]{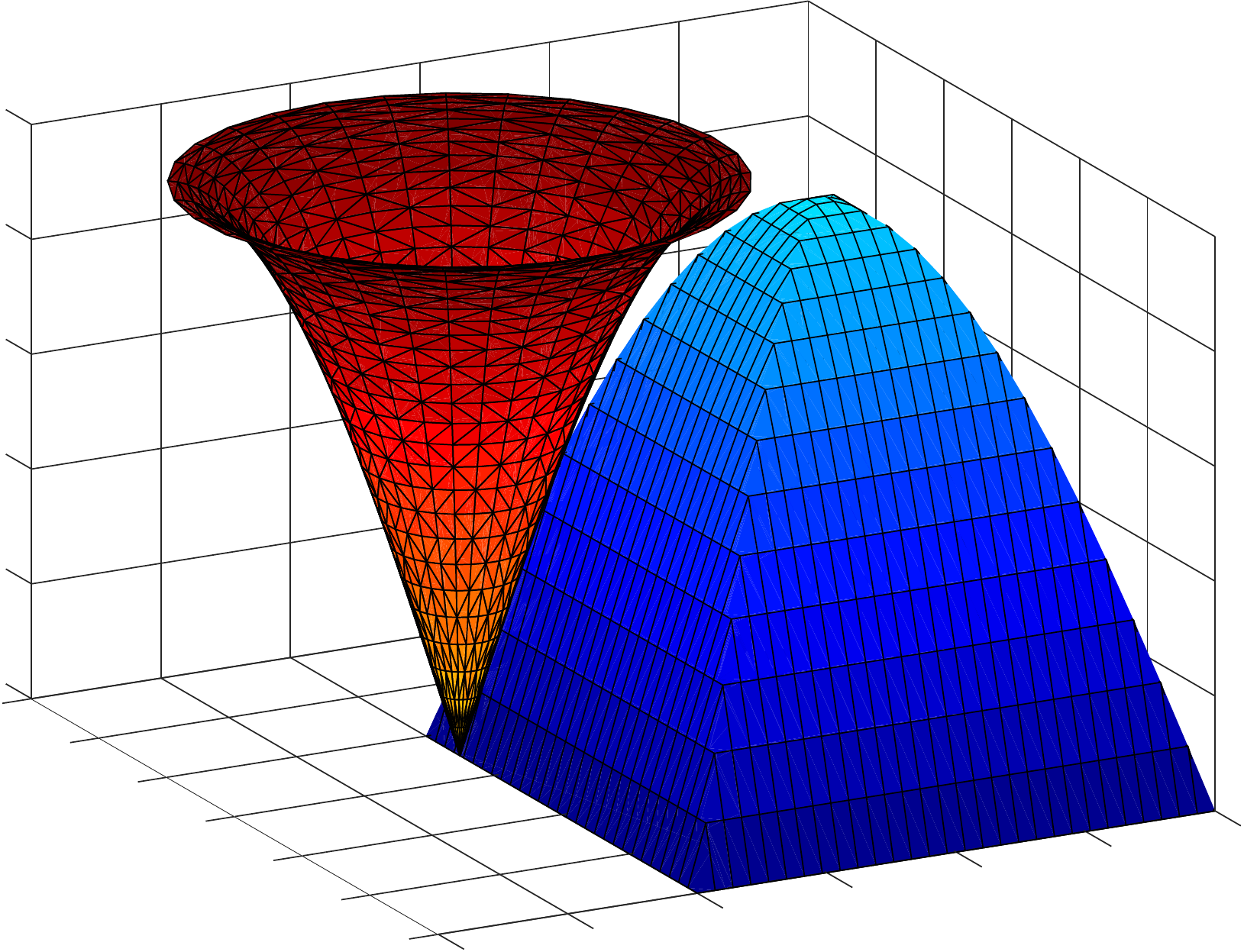}
    \caption{Functions $v$ (blue) and $\phi_{y_0}$ (red) defined in the proof of Theorem~\ref{thm.main} for a square.}
    \label{fig:proof}
\end{figure}

\begin{proof}
Let us consider the continuous function $v:\Omega\to\R$ given by
$$v(x)=\phi(\dist(x,\Omega^c)).$$
Since $\phi$ is positive outside the origin so is $v$ inside $\Omega$.
If we prove that $L v+\lambda v \leq 0$ for the given value of $\lambda$, we obtain the desired inequality.

Let us consider $x_0\in\Omega$ and $\psi \in C^2$ such that it touches $v$ at $x_0$ strictly from below.
Since $\Omega$ is an open set, there exists $y_0\in\partial\Omega$ such that ${\dist(x_0,\Omega^c)=\dist(x_0,y_0)}	$ and we can consider $\phi_{y_0}(x)=\phi(|x-y_0|)$ which is a continuous function defined in $B_r(y_0)$.
Then, since $\phi$ is radial increasing, one get $v\leq \phi_{y_0}$ and coincides with it at $x_0$.
So $\psi$ touches $\phi_{y_0}$ at $x_0$ strictly from below and hence $\psi$ satisfies the inequality.
This shows that $L v+\lambda v \leq 0$ in the viscosity sense as desired.
\end{proof}

\begin{remark}\rm
Given $r$ if we are able to construct $\phi$ for certain $\lambda(r)$ that depends continuously on $r$ since $\lambda_1(\Omega) \geq \lambda(r)$ for all $r>R$, we obtain ${\lambda_1(\Omega) \geq \lambda(R)}$.
\end{remark}

It may be the case that we could not construct $\phi$ as required above (see Example~\ref{nosol0}).
In that case we can modify our construction in order to obtain the lower bound as follows.
Given $\delta>0$, we consider 
\[
\Omega_\delta=\{x:\dist(x,\Omega)<\delta\}
\]
and
\[
R_\delta=\max_{x\in\bar\Omega_\delta}\dist(x,\Omega_\delta^c).
\]
\begin{theorem}
\label{thm.delta}
Suppose $\phi(r)$ is an increasing radial function defined in $B_{r}\setminus B_\delta$  for some $r>R_\delta$ with $\phi=0$ on $\partial B_\delta$ and $\lambda$  is such that
\[
L \phi + \lambda \phi \leq  0
\]
in $B_{r}\setminus B_\delta$. Then $\lambda_1(\Omega)\geq\lambda$.
\end{theorem}

\begin{proof}
The proof is completely analogous to that of Theorem~\ref{thm.main}.
We have to consider $v(x)=\phi(\dist(x,\Omega_\delta^c))$ which is positive in $\Omega$ since $\dist(x,\Omega_\delta^c)\geq \delta$ for all $x\in\Omega$.
And we prove that $v$ is a supersolution at $x_0$ by considering $\phi_{y_0}(x)=\phi(|x-y_0|)$ for 
$y_0\in\partial\Omega_\delta$ such that $\dist(x_0,\Omega_\delta^c)=\dist(x_0,y_0)$.
\end{proof}

Let us make some comments regarding $R_\delta$ which will be useful when applying Theorem~\ref{thm.delta}, see Example~\ref{nosol0}.
We observe that $R_\delta\geq R+\delta$ but equality is not true in general.
This can be seen by considering an U shaped domain, if $\delta$ is big enough, the `hole' inside the domain is covered and then $R_\delta$ is strictly bigger than $R+\delta$, see Figure~\ref{fig:Ushape}.
Let us prove that the equality holds for convex domains.

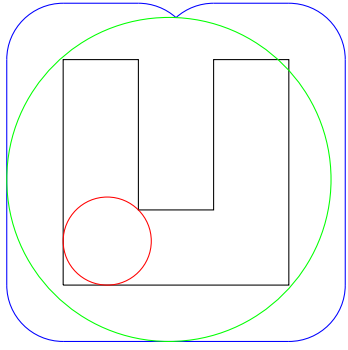
\begin{figure}
    \centering
\begin{tikzpicture}
\draw [blue](0-0.75,0) arc (180:270:0.75);
\draw [blue](3,0-0.75) arc (270:360:0.75);
\draw [blue](3+0.75,3+0) arc (0:90:0.75);
\draw [blue](2,3+0.75) arc (90:{90+asin(0.5/0.75)}:0.75);
\draw [blue](1+0.5,3+0.559) arc ({90-asin(0.5/0.75)}:90:0.75);
\draw [blue](0,3+0.75) arc (90:180:0.75);
\draw  (0,0) -- (3,0) -- (3,3) -- (2,3) -- (2,1) -- (1,1) -- (1,3) -- (0,3) -- (0,0);
\draw [blue](0,-0.75) -- (3,-0.75);
\draw [blue](3+0.75,0) -- (3+0.75,3);
\draw [blue](3,3+0.75) -- (2,3+0.75);
\draw [blue](1,3+0.75) -- (0,3+0.75);
\draw [blue](-0.75,3) -- (-0.75,0);
\draw [red] ({1/(1+1/sqrt(2))},{1/(1+1/sqrt(2))}) circle ({1/(1+1/sqrt(2))});
\def\r{(0.75+1.5)+(sqrt(0.75^2-0.5^2)+0.75+3)-sqrt(2*(0.75+1.5)*(sqrt(0.75^2-0.5^2)+0.75+3))}
\draw [green] ({\r-0.75},{\r-0.75}) circle ({\r});
\end{tikzpicture}
\caption{In black a U shaped domain ($\Omega$), in red the ball of maximum radius ($R$) included in $\Omega$, in blue $\Omega_\delta$ and in green the ball of maximum radius ($R_\delta$) included in $\Omega_\delta$, we have $R_\delta>R+\delta$.}
\label{fig:Ushape}
\end{figure}

\begin{lemma}
\label{lemmaRd}
When $\Omega$ is convex, $R_\delta=R+\delta$.
\end{lemma}

\begin{proof}
Let $y\in\Omega_\delta$ and $\tilde R>0$ such that $B_{\tilde{R}}(y)\subset\Omega_\delta$. 
Let us show that $B_{\tilde{R}-\delta}(y)\subset\Omega$, and hence $R_\delta-\delta\leq R$ as desired. 

Suppose not, let $x\in B_{\tilde{R}-\delta}(y)\setminus \Omega$.
As $x\not\in\Omega$ and $\Omega$ is convex there exists a hyperplane though $x$ such that one of the half-spaces defined by this hyperplane is disjoint with $\Omega$.
Now, points in that half-space at distance greater that $\delta$ from the hyperplane are not in $\Omega_\delta$ but this is a contradiction since $B_{\tilde{R}-|x-y|}\subset\Omega_\delta$ and $\tilde{R}-|x-y|>\delta$.
\end{proof}

\begin{remark}\rm
We have considered the Dirichlet eigenvalue problem given by $$Lu+\lambda u=0$$ but we can consider a more general version of the problem $Lu+\lambda Mu=0$, where $M$ is a given differential operator, or even more generally $$G(D^2 u, \nabla u, u, \lambda)=0.$$
As examples of this general situation we can consider $M=|u|^{\alpha}u$ as in \cite{birindelli2006} and $G(D^2 u, \nabla u, u, \lambda)=\min\{-\Delta_\infty u, |\nabla u|-\lambda u\}$ as in \cite{juutinen1999eigenvalue}.
Theorems~\ref{thm.main} and \ref{thm.delta} also hold in this more general case.
\end{remark}

\section{Examples}
\label{Sect.examples}
\setcounter{equation}{0}

In this section we compute explicitly the bound for the principal eigenvalue of the homogeneous infinity laplacian, the homogeneous $p$-laplacian, the eigenvalue problem that rises when considering the limit as $p\to \infty$ of the problem for the $p$-laplacian and Pucci extremal operator.
We denote $\lambda_{1,\infty}$ and $\lambda_{1,p}$ the principal eigenvalue of the homogeneous infinity laplacian and the homogeneous $p$-laplacian, respectively.
For the homogeneous infinity laplacian we prove that the principal eigenvalue is given by $\lambda_{1,\infty}=\left(\frac{\pi}{2R}\right)^2$.
For the homogeneous $p$-laplacian our bound allows us to prove that $\lim_{p\to\infty} \lambda_{1,p}= \lambda_{1,\infty}$, see \cite{martinez2014limit} for a different proof of this result.

\begin{example}\rm
\label{inflap}
Here we consider the homogeneous infinity laplacian, which is given by
\[
\Delta_\infty^H u=\left(\frac{\nabla u}{|\nabla u|}\right)^t D^2 u \frac{\nabla u}{|\nabla u|}.
\]
The eigenvalue problem for this operator was studied in \cite{juutinen2007principal}.
We want to prove that
\[
\lambda_{1,\infty}(\Omega)= \left(\frac{\pi}{2R}\right)^2,
\]
which gives us an explicit new characterization of the eigenvalue.

On the one hand we have that $\lambda_{1,\infty}(B_R)= \left(\frac{\pi}{2R}\right)^2$.
It is easy to check that
\[
u(x)=\sin \left( \frac{(R-||x||)\pi}{2R}\right)
\]
is the corresponding eigenfunction.
On the other hand it is easy to verify that
\[
\phi(x)=\sin \left( \frac{||x||\pi}{2R}\right)
\]
satisfies $L \phi + \lambda_{1,\infty} \phi \leq  0$ in $B_R\setminus\{0\}$, it is radially increasing in $B_R$ and $\phi(0)=0$.
Hence Theorem~\ref{thm.main} allows us to conclude the desired result.

Moreover $v(x)=\phi(\dist(x,\Omega^c))$ is an eigenfunction for stadium like domains.
As can be seen in the proof of Theorem~\ref{thm.main}, it is a supersolution to the equation.
In the same way it can be shown that it is a subsolution by considering the eigenfunction in balls of radius $R$ contained in $\Omega$.
Let us mention that in \cite{crasta2016characterization} stadium like domains are characterized by considering a Serrin-type problem for the homogeneous infinity laplacian.

\end{example}

\begin{example}\rm
We consider the homogeneous $p$-laplacian, that is
\[
\Delta_p^H u= \frac{1}{p}|\nabla u|^{2-p}div(|\nabla u|^{p-2} \nabla u)= \frac{p-2}{p}\Delta_\infty^H u + \frac{1}{p} \Delta u.
\]

When we look for radial solutions to the equation $\Delta_p^H v+\lambda v=0$ in $B_R$, we obtain the equation
\begin{equation}
\label{ODE}
v_{rr}+\frac{n-1}{p-1}\frac{v_r}{r}+\frac{p\lambda}{p-1}v=0.
\end{equation}

The general solution is given by
\[
v(r)= c_1 r^\alpha J_{\alpha}(\eta r) + c_2 r^\alpha Y_{\alpha}(\eta r),
\]
where 
\[
\alpha=\frac{1-\frac{n-1}{p-1}}{2}=\frac{p-n}{2(p-1)} \quad \text{ , } \quad \eta=\sqrt{\lambda\frac{p}{p-1}}
\]
and $J_{\alpha}$  and $Y_{\alpha}$ are Bessel functions.

In \cite{kawohl2014radial} the eigenvalue for a ball $B_R$ is computed, 
\[
\lambda_p(B_R)=\frac{p-1}{p}\left(\frac{\mu^{(-\alpha)}_1}{R}\right)^2,
\]
where $\mu^{(-\alpha)}_1$ is the first zero of the Bessel function $J_{-\alpha}$.
This implies that
\[
\lambda_{1,p}(\Omega)\leq 
\frac{p-1}{p}\left(\frac{\mu^{(-\alpha)}_1}{R}\right)^2.
\]

We want to construct an appropriate function to apply Theorem~\ref{thm.main}.
We consider the case $p>n$ (we analyse the case $p\leq n$ in the following example).
We observe that
\[
0<\alpha=\frac{p-n}{2(p-1)}.
\]
As we require $v(0)=0$, we have to take
\[
v(r)= c r^\alpha J_{\alpha}(\eta r).
\]
Then
\[
v'(r)= c \eta r^\alpha J_{\alpha-1}(\eta r)
\] 
and we can take $v$ increasing up to the first zero of the derivative.
We impose $v'(R)=0$, that is
\[
\frac{\mu^{(\alpha-1)}_1}{R}= \eta=\sqrt{\lambda\frac{p}{p-1}}.
\]
Then,
\[
\frac{p-1}{p}\left(\frac{\mu^{(\alpha-1)}_1}{R}\right)^2
\leq \lambda_{1,p}(\Omega),
\]
and we have that
\[
\frac{p-1}{p}\left(\frac{\mu^{(\alpha-1)}_1}{R}\right)^2
\leq \lambda_{1,p}(\Omega)\leq 
\frac{p-1}{p}\left(\frac{\mu^{(-\alpha)}_1}{R}\right)^2.
\]

Now let us consider the limit as $p\to\infty$.
Since $\alpha\to\frac{1}{2}^-$, we have that 
\[
\alpha -1\to -\frac{1}{2}^- \text{ and } -\alpha\to -\frac{1}{2}^+
\]
and hence
\[
\frac{p-1}{p}\left(\frac{\mu^{(\alpha-1)}_1}{R}\right)^2\to \left(\frac{\pi}{2R}\right)^2
\quad \text{ and } \quad
\frac{p-1}{p}\left(\frac{\mu^{(-\alpha)}_1}{R}\right)^2\to \left(\frac{\pi}{2R}\right)^2.
\]
We have proved that
\[
\lim_{p\to\infty} \lambda_{1,p}(\Omega)= \left(\frac{\pi}{2R}\right)^2 = \lambda_{1,\infty}(\Omega).
\]

\end{example}

\begin{example}\rm
\label{nosol0}
If we consider the case $p\leq n$ in the previous example, the ordinary differential equation \eqref{ODE} has no non-trivial solution with $v(0)=0$.
Hence we apply Theorem~\ref{thm.delta}.

We can take
\[
v(r)= c r^\alpha J_{\alpha}(\eta r).
\]
Then
\[
v'(r)= c \eta r^\alpha J_{\alpha-1}(\eta r).
\]
If $x$ is a zero of $J_{\alpha}$ and $y$ the first zero of $J_{\alpha-1}$ after $x$, we can choose the sign of $c$ such that $v$ is an increasing positive function in the interval $(x/\eta,y/\eta)$.
Let us assume that we can choose $\delta$ such that $\delta/R_\delta=x/y$ (we can do this when $\Omega$ is convex and hence $R_\delta=R+\delta$ as stated in Lemma~\ref{lemmaRd}). 
If we take $\eta=x/\delta=y/R_\delta$ we obtain that $v$ is an increasing positive function in the interval $(\delta,R_\delta)$ and we can apply Theorem~\ref{thm.delta}.
In the case that $R_\delta=R+\delta$, $\delta/R_\delta=x/y$ implies that $\delta=\frac{Rx}{y-x}$.
Then $\eta=\frac{y-x}{R}$, and we obtain
\[
\frac{p-1}{p}\left(\frac{y-x}{R}\right)^2
\leq \lambda_{1,p}(\Omega).
\]
Let us observe that the same can be done with $Y_{\alpha}$ instead of $J_{\alpha}$.

Let us make some explicit computation in a particular case, for the Laplacian in dimension 3.
We avoid the term $1/p$ in the operator and consider the equation $\Delta u+\lambda u=0$ in $\Omega\subset\R^3$.
We have $\alpha=-1/2$,
\[
x^{-1/2}J_{-1/2}(x)=\sqrt{\frac{2}{\pi}} \frac{cos(x)}{x}.
\]
The distance between the zeros of the function and the subsequent zero of its derivative increases and approaches $\pi/2$.
Hence, we obtain
\[
\left(\frac{\pi}{2R}\right)^2
\leq \lambda_{1,2}(\Omega).
\]

Let us compare our result to the classical Rayleigh-Faber-Krahn inequality which states
\[
\lambda_1(\Omega)\geq|\Omega|^{-\frac{2}{n}}C_n^{\frac{2}{n}}\left(\mu_1^{(\frac{n}{2}-1)}\right)^2 
\]
where $C_n$ is the volume of the $n$-dimensional unit ball and
 $\mu^{(\alpha)}_1$ is the first zero of the Bessel function $J_{\alpha}$.
This inequality is sharp for the unit ball, in $\R^3$ we have
\[
\lambda_{1}(B_1)=\left(\frac{\pi}{R}\right)^2.
\]
If $|\Omega|\geq 8|B_R|$, we have
\[
|\Omega|^{-\frac{2}{3}}C_3^{\frac{2}{3}}\left(\mu_1^{\frac{3}{2}-1}\right)^2  \leq
|8B_R|^{-\frac{2}{3}}|B_1|^{\frac{2}{3}}\left(\mu_1^{(\frac{1}{2})}\right)^2=
\left(\frac{\pi}{2R}\right)^2,
\]
where we have used that $\mu_1^{(\frac{1}{2})}=\pi$.
Hence, we have obtained that our inequality is sharper in this case. 
This holds, for example, for a cylinder tall enough.

\end{example}

\begin{example}\rm
We consider the equation
\[
\min\{-\Delta_\infty u, |\nabla u|-\lambda u\}=0,
\]
where 
\[
\Delta_\infty u=\left(\nabla u\right)^t D^2 u \nabla u
\]
is the infinity laplacian.
This equation arises when considering the limit as $p\to \infty$ in the eigenvalue problem for the $p$-laplacian, see \cite{juutinen1999eigenvalue}.

In this case the principal eigenvalue is $\frac{1}{R}$, we can prove this fact in the same way as in Example~\ref{inflap} by considering
$u(x)= R-||x||$ and $\phi(x)=||x||$.

\end{example}

\begin{example}\rm
We consider Pucci's extremal operator, that is
\[
M_{\gamma,\Gamma}^+(D^2u)=\Gamma \sum_{e_i>0} e_i+ \gamma \sum_{e_i<0} e_i,
\]
where $e_i$ are the eigenvalues of $D^2u$.

When $u(x)=\phi(r)$ is radial, the eigenvalues are $\phi''(r)$ with multiplicity one and $\phi'(r)/r$ with multiplicity $n-1$.
Since we require the function to be increasing, we have $\phi'(r)/r>0$.
If we attempt to find a solution with $\phi''(r)\leq 0$,
we obtain the equation
\[
\phi''+\Gamma(n-1){\gamma}\frac{\phi'}{r}+\frac{\lambda}{\gamma}\phi=0.
\]
As we want $\phi, \phi'\geq 0$ for those solutions we will have $\phi''(r) \leq 0$ as desired.

Again, the general solution is given by
\[
v(r)= c_1 r^\alpha J_{\alpha}(\eta r) + c_2 r^\alpha Y_{\alpha}(\eta r),
\]
where 
\[
\alpha=\frac{1-\frac{\Gamma(n-1)}{\gamma}}{2}=\frac{\gamma-\Gamma(n-1)}{2\gamma} \quad \text{ and } \quad \eta=\sqrt{\frac{\lambda}{\gamma}}.
\]

We can obtain the bound as in the previous examples.
Let us illustrate this with a particular case.
With $\gamma=1$, $\Gamma=2$ in dimension 2 we have $\alpha=-1/2$ as in the end of Example~\ref{nosol0}, we obtain 
$
\left(\frac{\pi}{2R}\right)^2
\leq \lambda_{1}(\Omega)
$.

\end{example}

{\bf Ackledgements} This work has been partially supported by Consejo Nacional de
Investigaciones Cient\'ificas y T\'ecnicas (CONICET-Argentina).

\bibliography{eigen_arxiv2}
\bibliographystyle{plain}
\end{document}